\documentclass{amsart}
\usepackage[all]{xy}
\usepackage{amssymb}
\usepackage{amsthm}
\usepackage{amsmath}
\usepackage{tikz}
\usetikzlibrary{shapes.geometric}
\usepackage[justification=centering]{caption}
\numberwithin{equation}{section}
\usepackage{yfonts}

\newtheorem{theorem}{Theorem}[section]
\theoremstyle{remark}
\newtheorem{remark}[theorem]{Remark}
\theoremstyle{plain}
\newtheorem{lemma}[theorem]{Lemma}
\theoremstyle{plain}

\theoremstyle{definition}
\newtheorem{definition}[theorem]{Definition}
\title{A Multiplication Formula for the Modified Caldero-Chapoton Map}

\begin{document}
\author{David Pescod}
\address{School of Mathematics and Statistics, Newcastle University, Newcastle Upon Tyne NE1 7RU, United Kingdom}
\email{d.pescod1@ncl.ac.uk}

\begin{abstract}
A frieze in the modern sense is a map from the set of objects of a triangulated category $\mathsf{C}$ to some ring. A frieze $X$ is characterised by the property that if $\tau x\rightarrow y\rightarrow x$ is an Auslander-Reiten triangle in $\mathsf{C}$, then $X(\tau x)X(x)-X(y)=1$. The canonical example of a frieze is the (original) Caldero-Chapoton map, which send objects of cluster categories to elements of cluster algebras.
\par
In \cite{friezes1} and \cite{friezes2}, the notion of generalised friezes is introduced. A generalised frieze $X'$ has the more general property that $X'(\tau x)X'(x)-X'(y)\in\{0,1\}$. The canonical example of a generalised frieze is the modified Caldero-Chapoton map, also introduced in \cite{friezes1} and \cite{friezes2}.
\par
Here, we develop and add to the results in \cite{friezes2}. We define Condition F for two maps $\alpha$ and $\beta$ in the modified Calero-Chapoton map, and in the case when $\mathsf{C}$ is 2-Calabi-Yau, we show that it is sufficient to replace a more technical ``frieze-like" condition from \cite{friezes2}. We also prove a multiplication formula for the modified Caldero-Chapoton map, which significantly simplifies its computation in practice.
\end{abstract}

\maketitle
\section{Introduction}\label{introduction}
\subsection{Summary}
This paper focuses around two main topics: generalised friezes with integer values (see \cite{friezes1}) and generalised friezes taking values inside a Laurent polynomial ring (see \cite{friezes2}).
\par
A frieze is a map $X:\mathrm{obj}\,\mathsf{C}\rightarrow A$, where $\mathsf{C}$ is some triangulated category with Auslander-Reiten (AR) triangles and $A$ is a ring, such that the following exponential conditions are satisfied:
\begin{equation}\label{eqn:friezeexponential}
X(0)=1~\mathrm{and}~X(a\oplus b)=X(a)X(b),
\end{equation}
and if $\tau x\rightarrow y\rightarrow x$ is an AR triangle, then 
\begin{equation}\label{intro:def:frieze}
X(\tau x)X(x)-X(y)=1.
\end{equation}
The canonical example of a frieze is the Caldero-Chapoton map, which we recall in Section \ref{sec:theccmap}.
\par
Generalised friezes are similarly defined maps $X':\mathrm{obj}\,\mathsf{C}\rightarrow A$, also satisfying the exponential conditions in (\ref{eqn:friezeexponential}), however we permit the more general property that
\begin{equation}\label{intro:def:genfrieze}
X'(\tau x)X'(x)-X'(y)\in\{0,1\}.
\end{equation}
The canonical example of a generalised frieze is the modified Caldero-Chapoton map, which we recall in Section \ref{sec:modccmap}. The arithmetic version $\pi$, with integer values, is defined in Equation (\ref{intro:def:modccmaparithmetic}), whilst the more general version $\rho$, taking values inside a Laurent polynomial ring, is defined in Equation (\ref{eqn:introccmap}).
\par
The modified Caldero-Chapoton map was introduced in \cite{friezes2}, and we improve and add to the results of that paper. When working with a 2-Calabi-Yau category, we manage to replace the technical ``frieze-like" condition (see \cite[def.~1.4]{friezes2}) for the maps $\alpha$ and $\beta$ in the generalised Caldero-Chapoton map (Equation (\ref{eqn:introccmap})),  by our so-called Condition F (see Definition \ref{def:conditionf}). This condition significantly simplifies the frieze-like condition and demonstrates the roles of $\alpha$ and $\beta$. We will see that $\alpha$ plays the role of a ``generalised index", whilst $\beta$ provides a correction term to $\alpha$ being exponential over a distinguised triangle.
\par
We use this to establish a multiplication formula for the modified Caldero-Chapoton map $\rho$ (see Theorem \ref{thm:extensionadaption}), allowing its computation in practice. In \cite{friezes2}, the computation of $\rho$ is not addressed. However, our multiplication formula does address the computation, and does so in a simpler manner than merely applying the definition. In particular, the formula allows us to compute values of $\rho$ without calculating Euler characteristics of submodule Grassmannians which are otherwise part of the definition of $\rho$.

\subsection{Cluster Categories}
Cluster categories were first introduced in \cite{clustercombinatorics} by Buan-Marsh-Reineke-Reiten-Todorov as a means of understanding the `decorated quiver representations' introduced by Reineke-Marsh-Zelevinksy in \cite{mrz}. Let $Q$ be a finite quiver with no loops or cycles, and consider the category $\mathsf{mod}\,\mathbb{C}Q$ of finitely generated modules over the path algebra $\mathbb{C}Q$. Then, set
\[
\mathsf{D}(Q)=\mathsf{D}^b(\mathsf{mod}\,\mathbb{C}Q),
\]
 the bounded derived category of $\mathsf{mod}\,\mathbb{C}Q$.
\par
The cluster category of type $Q$, denoted $\mathsf{C}(Q)$, is defined to be the orbit category of $\mathsf{D}(Q)$ under the action of the cyclic group generated by the autoequivalence $\tau^{-1}\Sigma=S^{-1}\Sigma^2$, where $\tau$ is the Auslander-Reiten translation, $\Sigma$ the suspension functor and $S$ the Serre functor. That is, 
\[
\mathsf{C}(Q)=\mathsf{D}(Q)/(S^{-1}\Sigma^2).
\]
The objects in $\mathsf{C}(Q)$ are the same as those in $\mathsf{D}(Q)$, however, the morphism sets in $\mathsf{C}(Q)$ are given by
\[
\textnormal{Hom}_{\mathsf{C}(Q)}(X,Y)=\bigoplus_{n\in\mathbb{Z}} \textnormal{Hom}_{\mathsf{D}(Q)}(X,(S^{-1}\Sigma^2)^nY).
\]
We note that $\mathsf{C}(Q)$ possesses a triangulated structure, it is $\mathbb{C}$-linear, Hom-finite, Krull-Schmidt and 2-Calabi-Yau, meaning that its Serre functor is $\Sigma^2$. It is also essentially small and has split idempotents. See \cite[sec.~1]{clustercombinatorics}
\par
The cluster category of Dynkin type $A_n$, denoted by $\mathsf{C}(A_n)$, has a very nice polygon model associated to it. This is due to Caldero, Chapoton and Schiffler in \cite{clusterancase}, who defined, for finite quivers of type $A_n$, an equivalent category to the cluster category in \cite{clustercombinatorics} in a totally different manner. This is done using a triangulation of a regular $(n+3)$-gon $P$, with objects and morphisms described in \cite[sec.~2]{clusterancase}.
\par
 The category $\mathsf{C}(A_n)$ carries several nice properties.
 There is a bijection between the set of indecomposables of $\mathsf{C}(A_n)$, denoted indec$\,\mathsf{C}(A_n)$, and the set of diagonals of $P$. We also identify each edge of $P$ with the zero object in $\mathsf{C}(A_n)$. Applying the suspension functor $\Sigma$ to an indecomposable corresponds to rotating the endpoints of the associated diagonal one vertex clockwise. That is, if the vertices of $P$ are labelled in an anticlockwise fashion with the set $\{0,1,\dots,n+2\}$, then for some indecomposable $\{i,j\}$, where $i$ and $j$ are vertices of $P$, we have
\[
\Sigma\{i,j\}=\{i-1,j-1\}.
\]
Such coordinates should clearly be taken modulo $n+3$.
\par
Identifying indecomposables of $\mathsf{C}(A_n)$ with the diagonals of $P$ carries the convenient property that for $a,b\in\,$indec$\,\mathsf{C}(A_n)$,
\[
\textnormal{dim}_\mathbb{C}\textnormal{Ext}_{\mathsf{C}(A_n)}^1(a,b)=
\left\{
\begin{array}{ll}
1, & \textnormal{if~$a$~and~$b$~cross}\\
0, & \mathrm{if}~a~\textnormal{and~$b$~do not~cross}.
\end{array}
\right.
\]

\subsection{The Auslander-Reiten Quiver}
The Auslander-Reiten quiver for $\mathsf{C}(A_n)$ is $\mathbb{Z}A_n$ modulo a glide reflection. A coordinate system may be put on the quiver, matching up with the diagonals of the $(n+3)$-gon (see Figure \ref{fig:arquiverca5} for an example of $\mathsf{C}(A_5)$).

\begin{figure}
\centering
\begin{tikzpicture}[scale=1, every node/.style={font=\footnotesize}]
\node at (0,0) {\{1,3\}};
\node at (2,0) {\{2,4\}};
\node at (4,0) {\{3,5\}};
\node at (6,0) {\{4,6\}};
\node at (8,0) {\{5,7\}};

\node at (1,1) {\{1,4\}};
\node at (3,1) {\{2,5\}};
\node at (5,1) {\{3,6\}};
\node at (7,1) {\{4,7\}};

\node at (0,2) {\{4,8\}};
\node at (2,2) {\{1,5\}};
\node at (4,2) {\{2,6\}};
\node at (6,2) {\{3,7\}};
\node at (8,2) {\{4,8\}};

\node at (1,3) {\{5,8\}};
\node at (3,3) {\{1,6\}};
\node at (5,3) {\{2,7\}};
\node at (7,3) {\{3,8\}};

\node at (0,4) {\{5,7\}};
\node at (2,4) {\{6,8\}};
\node at (4,4) {\{1,7\}};
\node at (6,4) {\{2,8\}};
\node at (8,4) {\{1,3\}};

\draw [->] (0.3,0.3) --(0.7,0.7);
\draw [->] (1.3,1.3) --(1.7,1.7);
\draw [->] (2.3,2.3) --(2.7,2.7);
\draw [->] (3.3,3.3) --(3.7,3.7);

\draw [->] (2.3,0.3) --(2.7,0.7);
\draw [->] (3.3,1.3) --(3.7,1.7);
\draw [->] (4.3,2.3) --(4.7,2.7);
\draw [->] (5.3,3.3) --(5.7,3.7);

\draw [->] (4.3,0.3) --(4.7,0.7);
\draw [->] (5.3,1.3) --(5.7,1.7);
\draw [->] (6.3,2.3) --(6.7,2.7);
\draw [->] (7.3,3.3) --(7.7,3.7);

\draw [->] (6.3,0.3) --(6.7,0.7);
\draw [->] (7.3,1.3) --(7.7,1.7);

\draw [->] (0.3,2.3) --(0.7,2.7);
\draw [->] (1.3,3.3) --(1.7,3.7);

\draw [->] (0.3,3.7) --(0.7,3.3);
\draw [->] (1.3,2.7) --(1.7,2.3);
\draw [->] (2.3,1.7) --(2.7,1.3);
\draw [->] (3.3,0.7) --(3.7,0.3);

\draw [->] (2.3,3.7) --(2.7,3.3);
\draw [->] (3.3,2.7) --(3.7,2.3);
\draw [->] (4.3,1.7) --(4.7,1.3);
\draw [->] (5.3,0.7) --(5.7,0.3);

\draw [->] (4.3,3.7) --(4.7,3.3);
\draw [->] (5.3,2.7) --(5.7,2.3);
\draw [->] (6.3,1.7) --(6.7,1.3);
\draw [->] (7.3,0.7) --(7.7,0.3);

\draw [->] (6.3,3.7) --(6.7,3.3);
\draw [->] (7.3,2.7) --(7.7,2.3);

\draw [->] (0.3,1.7) --(0.7,1.3);
\draw [->] (1.3,0.7) --(1.7,0.3);

\draw [dotted] (0,0.25) --(0,1.75);
\draw [dotted] (0,2.25) --(0,3.75);
\draw [dotted] (8,0.25) --(8,1.75);
\draw [dotted] (8,2.25) --(8,3.75);
\end{tikzpicture}
\caption{The Auslander-Reiten quiver for $\mathsf{C}(A_5)$. The dotted lines are identified with opposite orientations.}
\label{fig:arquiverca5}
\end{figure}
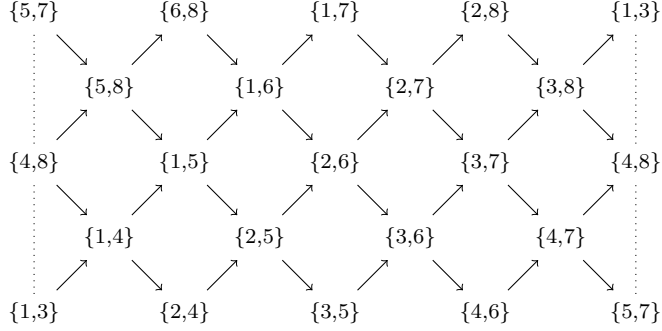

The Auslander-Reiten triangles in $\mathsf{C}(A_n)$ take the form
\[
\{i-1,j-1\}\rightarrow\{i-1,j\}\oplus\{i,j-1\}\rightarrow\{i,j\},
\]
where if $\{i-1,j\}$ or $\{i,j-1\}$ correspond to an edge of $P$, then they should be taken to be zero, see \cite[lem.~3.15]{brustlezhang}. Notice that each AR triangle can be realised from a diamond in the AR quiver. That is, if
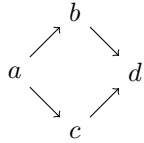
\[
\begin{tikzpicture}[scale=0.8]
\node (C) at (1,0) {$c$};
\node (A) at (0,1) {$a$};
\node (D) at (2,1) {$d$};
\node (B) at (1,2) {$b$};

\draw [->, to path={--(\tikztotarget)}] (A) edge (C) (C) edge (D) (A) edge (B) (B) edge (D);
\end{tikzpicture}
\]
is a diamond inside the AR quiver, then
\[
a\rightarrow b\oplus c\rightarrow d,
\]
is an AR triangle. If $a$ and $d$ sit on the upper boundary of the AR quiver, then $b$ should be taken as zero, whereas if $a$ and $d$ sit on the lower boundary, then $c$ is taken to be zero. Note that a frieze $X$ on $\mathsf{C}(A_n)$ satisfies:
\begin{equation}\label{canfrieze}
X(a)X(d)-X(b)X(c)=1
\end{equation}

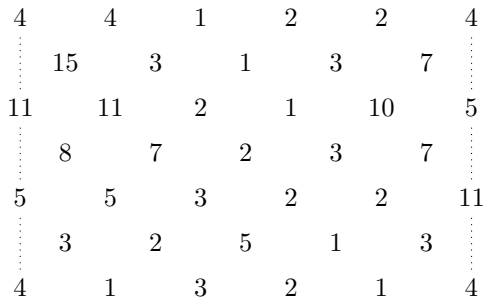
\begin{figure}
\centering
\begin{tikzpicture}[scale=0.6]
\node at (0,0) {4};
\node at (2,0) {1};
\node at (4,0) {3};
\node at (6,0) {2};
\node at (8,0) {1};
\node at (10,0) {4};

\node at (1,1) {3};
\node at (3,1) {2};
\node at (5,1) {5};
\node at (7,1) {1};
\node at (9,1) {3};

\node at (0,2) {5};
\node at (2,2) {5};
\node at (4,2) {3};
\node at (6,2) {2};
\node at (8,2) {2};
\node at (10,2) {11};

\node at (1,3) {8};
\node at (3,3) {7};
\node at (5,3) {2};
\node at (7,3) {3};
\node at (9,3) {7};

\node at (0,4) {11};
\node at (2,4) {11};
\node at (4,4) {2};
\node at (6,4) {1};
\node at (8,4) {10};
\node at (10,4) {5};

\node at (1,5) {15};
\node at (3,5) {3};
\node at (5,5) {1};
\node at (7,5) {3};
\node at (9,5) {7};

\node at (0,6) {4};
\node at (2,6) {4};
\node at (4,6) {1};
\node at (6,6) {2};
\node at (8,6) {2};
\node at (10,6) {4};

\draw[dotted] (0,0.4) --(0,1.6);
\draw[dotted] (0,2.4) --(0,3.6);
\draw[dotted] (0,4.4) --(0,5.6);
\draw[dotted] (10,0.4) --(10,1.6);
\draw[dotted] (10,2.4) --(10,3.6);
\draw[dotted] (10,4.4) --(10,5.6);
\end{tikzpicture}
\caption{A frieze on the cluster category of Dynkin type $A_7$. The dotted lines are identified with opposite orientations.}
\label{fig:ca7frieze}
\end{figure}

\subsection{The Caldero-Chapoton Map}\label{sec:theccmap}
The (original) Caldero-Chapoton map is a map which sends certain (so-called reachable) indecomposable objects of a cluster category to cluster variables of the corresponding cluster algebra, see \cite[sec.~4.1]{calderokeller}. The map, which we denote by $\gamma$, depends on a cluster tilting object $T$ inside the cluster category and makes precise the idea that the cluster category is a categorification of the cluster algebra. It is required that the category is 2-Calabi-Yau (for example a cluster category), and it is a well known property of $\gamma$ that it is a frieze (see \cite[def.~1.1]{assemdupont}, \cite[prop.~3.10]{hallalgebras}, \cite[thm.]{dominguezgeiss}). That is, the Caldero-Chapoton map is a map $\gamma\,:\,$obj$\,\mathsf{C}\rightarrow A$, where $A$ is a certain ring, which satisfies the property in Equation (\ref{intro:def:frieze}), as well as the exponential conditions in (\ref{eqn:friezeexponential}).

\subsection{Frieze Patterns}
Frieze patterns were first introduced by Conway and Coxeter in \cite{conwaycoxeter1} and \cite{conwaycoxeter2}. An example of such a frieze pattern, known as a Conway-Coxeter frieze, is given in Figure \ref{fig:ca7frieze}. This frieze pattern is obtained from the original Caldero-Chapoton map $\gamma$ by omitting the arrows from the AR quiver of $\mathsf{C}(A_7)$ and replacing each vertex by the value of $\gamma$ applied to that indecomposable.
\par
In formal terms, for some positive integer $n$, a frieze pattern is an array of $n$ offset rows of positive integers. Each diamond
\[
\begin{tikzpicture}[scale=0.6]
\node (C) at (1,0) {$\delta$};
\node (A) at (0,1) {$\beta$};
\node (D) at (2,1) {$\eta$};
\node (B) at (1,2) {$\alpha$};
\end{tikzpicture}
\]
in a frieze pattern satisfies a so-called determinant property in that
\[
\beta\eta-\alpha\delta=1.
\]
If $\beta$ and $\eta$ are on the top row of the frieze, then the determinant property becomes
\[
\beta\eta-\delta=1,
\]
and if $\beta$ and $\eta$ are on the bottom row, the determinant property becomes
\[
\beta\eta-\alpha=1.
\]
Observe that the Caldero-Chapoton map satisfies these equations by virtue of Equation (\ref{canfrieze}), that is, because it is a frieze. Frieze patterns are known to be invariant under a glide reflection. A region of the frieze pattern, known as a fundamental domain, is enough to produce the whole frieze pattern by repeatedly performing a glide reflection.

\subsection{A Modified Caldero-Chapoton Map}\label{sec:modccmap}
We assume in the rest of the paper that $\mathsf{C}$ is an essentially small, $\mathbb{C}$-linear, Hom-finite, triangulated category, which is Krull-Schmidt and has AR triangles. In \cite{friezes1}, Holm and J{\o}rgensen introduce a modified version of the Caldero-Chapoton map, which we denote by $\pi$, that relies on a rigid object $R\in\,$obj$\,\mathsf{C}$, a much weaker condition than that of being a cluster tilting object. We say that an object $R$ is rigid if 
\[
\textnormal{Hom}_\mathsf{C}(R,\Sigma R)=0.
\]
We also note that $\pi$ does not require that the category is 2-Calabi-Yau, allowing us to work with a category $\mathsf{C}$ that is more general than a cluster category.
\par
Consider the endomorphism ring $E=$\,End$_\mathsf{C}(R)$, and define $\mathsf{mod}\,E$ to be the category of finite dimensional right $E$-modules. Then, there is a functor
\[
G\,:\,\mathsf{C}\rightarrow\mathsf{mod}E
\]
\vspace{-5mm}
\[
\hspace{2cm}c\mapsto\textnormal{Hom}_\mathsf{C}(R,\Sigma c).
\]
For some object $c\in\mathsf{C}$, the modified Caldero-Chapoton map is then defined by the formula:
\begin{equation}\label{intro:def:modccmaparithmetic}
\pi(c)=\chi(\textnormal{Gr}(Gc)),
\end{equation}
where Gr denotes the Grassmannian of submodules and $\chi$ is the Euler characteristic defined by cohomology with compact support (see \cite[p.~93]{Fulton}).
\par
It is proved in \cite{friezes1} that $\pi$ is a generalised frieze; that is, as well as the exponential properties given in Equation (\ref{eqn:friezeexponential}) it satisfies the property given in Equation (\ref{intro:def:genfrieze}). 
\par
 Define $\mathsf{R}=\,$add$\,R$, the full subcategory whose objects are finite direct sums of the summands of $R$ (see Section \ref{sec:introductory} for details of this setup). This full subcategory, which is clearly closed under direct sums and summands, is rigid in the sense that Hom$_\mathsf{C}(\mathsf{R},\Sigma\mathsf{R})=0$. A multiplication formula for computing $\pi$ is also proved in \cite{friezes1}.  Let $m\in\,$indec$\,\mathsf{C}$ and $r\in\,$indec$\,\mathsf{R}$ satisfy that Ext$_\mathsf{C}^1(m,r)$ and Ext$_\mathsf{C}^1(r,m)$ both have dimension one over $\mathbb{C}$. Then, there are nonsplit triangles
\begin{equation}\label{eqn:intrononsplitextensions}
\xymatrix{
m\ar[r]^\mu& a\ar[r]^-\gamma & r \ar[r]^-\delta &\Sigma m~,~r\ar[r]^-\sigma & b\ar[r]^\eta & m \ar[r]^-\zeta &\Sigma r,
}
\end{equation}
that are unique up to isomorphism. It is proved in \cite{friezes1} that 
\begin{equation}\label{eqn:originalextensionintro}
\pi(m)=\pi(a)+\pi(b).
\end{equation}
This formula can be applied iteratively to compute values of $\pi$.
\par
In \cite{friezes2} Holm and J{\o}rgensen redefine the modified Caldero-Chapoton map in a more general manner (the work in \cite{friezes1} is a special case of that in \cite{friezes2}). They define $\rho$ by
\begin{equation}\label{eqn:introccmap}
\rho(c)=\alpha(c)\sum_e\chi(\textnormal{Gr}_e(Gc))\beta(e).
\end{equation}
Here, the sum is taken over $e\in\,$K$_0(\mathsf{mod}\,E)$, the Grothendieck group of the abelian category $\mathsf{mod}\,E$, and Gr$_e(Gc)$ is the Grassmannian of $E$-submodules $M\subseteq Gc$ with K$_0$-class satisfying $[M]=e$. The maps
\begin{equation}\label{intro:def:alphabeta}
 \alpha:\,\mathrm{obj}\,\mathsf{C}\rightarrow A ~\mathrm{and}~ \beta:\,\mathrm{K}_0(\mathsf{mod}\,E)\rightarrow A 
\end{equation}
are both exponential maps in the sense that
\begin{equation}\label{eqn:alphabetaexponential}
\begin{gathered}
\alpha(0)=1~,~\alpha(x\oplus y)=\alpha(x)\alpha(y),\\
\beta(0)=1~,~\beta(e+f)=\beta(e)\beta(f),
\end{gathered}
\end{equation}
and $A$ is some commutative ring, see \cite[setup~1.1]{friezes2}.
\par
When the maps $\alpha$ and $\beta$ satisfy a technical ``frieze-like" condition, given in \cite[def.~1.4]{friezes2}, the map $\rho$ becomes a generalised frieze, as in Equation (\ref{intro:def:genfrieze}).

\subsection{This Paper}
In this paper, we show a simpler condition on $\alpha$ and $\beta$ than that in \cite{friezes2}, which implies that $\rho$ is a generalised frieze. We also show that a similar multiplication formula to that in Equation (\ref{eqn:originalextensionintro}) holds with $\rho$ instead of $\pi$. This permits a simpler iterative procedure for computing $\rho$ than the one given in \cite{friezes2}.
\par
We give the following definition:
\newline
\paragraph{\textbf{Definition}}\textbf{\ref{def:conditionf}} (Condition F).
\textit{
We say that the maps $\alpha$ and $\beta$ satisfy Condition F if for each triangle
\begin{equation*}
\xymatrix{
x \ar[r]^{\varphi} & y \ar[r]^{\omega} & z \ar[r]^{\psi} & \Sigma x
}
\end{equation*}
in $\mathsf{C}$ such that $Gx$, $Gy$ and $Gz$ have finite length in $\mathsf{Mod}E$, the following property holds:
\[
\alpha(y)=\alpha(x\oplus z)\beta([\textnormal{Ker}\,G\varphi]).
\]
}
\paragraph{}
The following main result shows that when $\mathsf{C}$ is 2-Calabi-Yau, this definition is sufficient to replace the frieze-like condition from \cite[def.~1.4]{friezes2}.
\newline
\paragraph{\textbf{Theorem}}\textbf{\ref{alphabetadef}}
\textit{
Assume that the exponential maps $\alpha : \textnormal{obj}\,\mathsf{C}\rightarrow A$ and $\beta : $K$_0(\mathsf{mod}E)\rightarrow A$ from Equation (\ref{intro:def:alphabeta}) satisfy Condition F from Definition~\ref{def:conditionf}. Then, the modified Caldero-Chapoton map $\rho$ of Equation (\ref{eqn:introccmap}) is a generalised frieze in the sense that it satisfies Equation (\ref{intro:def:genfrieze}), as well as the exponential conditions in (\ref{eqn:friezeexponential}).
\newline
}
\paragraph{}

We then proceed by proving in Lemma \ref{alphabeta} that the construction of $\alpha$ and $\beta$ in \cite[def.~2.8]{friezes2} satisfies Condition F. Then, Theorem \ref{alphabetadef}, together with Lemma \ref{alphabeta}, recovers a main result of \cite{friezes2}, proving that the construction of $\alpha$ and $\beta$ in \cite[def.~2.8]{friezes2} results in $\rho$ being a generalised frieze.

In Section \ref{extensionadaption}, we prove the multiplication formula for $\rho$, similar to the arithmetic case for $\pi$ in (\ref{eqn:originalextensionintro}). This multiplication formula is as follows:
\newline
\paragraph{\textbf{Theorem}}\textbf{\ref{thm:extensionadaption}}
\textit{
Let $m\in\textnormal{indec}\,\mathsf{C}$ and $r\in\textnormal{indec}\,\mathsf{R}$ such that $\textnormal{Ext}^{1}_{\mathsf{C}}(r,m)$ and $\textnormal{Ext}^{1}_{\mathsf{C}}(m,r)$ both have dimension one over $\mathbb{C}$. Then, there are nonsplit triangles 
\[
\xymatrix{
m \ar[r]^-\mu & a \ar[r]^-\gamma & r \ar[r]^-\delta & \Sigma m 
&and&
r \ar[r]^-\sigma & b \ar[r]^-\eta & m \ar[r]^-\zeta & \Sigma r,
}
\]
with $\delta$ and $\zeta$ nonzero. Let $Gm$ have finite length in $\mathsf{Mod}E$, then
\[
\rho(r)\rho(m)=\rho(a)+\rho(b).
\]
\newline
}
\paragraph{}
\vspace{-5mm}
This theorem can be applied inductively in order to simplify the computation of values of $\rho$. For some indecomposable $m$ in $\mathsf{C}$, one may find an $a$ and $b$ (which are the middle terms of the nonsplit extensions), such that $\rho(m)=\rho(a)+\rho(b)$. The theorem can then be reapplied to find each of $\rho(a)$ and $\rho(b)$. The process will eventually terminate at the stage where calculating $\rho$ of some indecomposable reduces to calculating the index of that indecomposable. Substituting back into the equation allows a simple calculation of $\rho(m)$.

We illustrate the procedure in Section \ref{sec:ca5example} by computing $\rho$ of an indecomposable in the Auslander-Reiten quiver for $\mathsf{C}(A_5)$. This retrieves one of the vertices of the AR quiver in Figure \ref{fig:ca5quiver}. Note that this example already appeared in \cite[sec.~3]{friezes2}, but Theorem \ref{thm:extensionadaption} makes our computation much simpler.

\begin{figure} 
\centering
\begin{tikzpicture}[scale=0.6]\hspace{8mm}
\hspace{-1cm}
\node (A) at (0,0) {$\frac{1+uv+vz}{v}$};
\node (B) at (0,4) {$\frac{1+uv+vz}{u}$};
\node (C) at (0,8) {$z$};
\node (D) at (2,2) {$1+uv+vz$};
\node (E) at (2,6) {$\frac{u+z}{u}$};
\node (F) at (4,0) {$v$};
\node (G) at (4,4) {$u+z$};
\node (H) at (4,8) {$\frac{u+z}{uz}$};
\node (I) at (6,2) {$1$};
\node (J) at (6,6) {$\frac{u+z}{z}$};
\node (K) at (8,0) {$\frac{1}{v}$};
\node (L) at (8,4) {$\frac{1}{z}$};
\node (M) at (8,8) {$u$};
\node (N) at (10,2) {$\frac{1+vz}{vz}$};
\node (O) at (10,6) {$1$};
\node (P) at (12,0) {$\frac{1+vz}{z}$};
\node (Q) at (12,4) {$\frac{1+vz}{v}$};
\node (R) at (12,8) {$\frac{1}{u}$};
\node (S) at (14,2) {$1+vz$};
\node (T) at (14,6) {$\frac{1+uv+vz}{uv}$};
\node (U) at (16,0) {$z$};
\node (V) at (16,4) {$\frac{1+uv+vz}{u}$};
\node (W) at (16,8) {$\frac{1+uv+vz}{v}$};

\draw[dotted, to path={--(\tikztotarget)}] (A) edge (B) (B) edge (C) (U) edge (V) (V) edge (W);

\draw [->, to path={--(\tikztotarget)}] (A) edge (D) (D) edge (G) (G) edge (J) (J) edge (M) (B) edge (E) (E) edge (H) (C) edge (E) (E) edge (G) (B) edge (D) (H) edge (J) (F) edge (I) (I) edge (L) (L) edge (O) (O) edge (R) (D) edge (F) (G) edge (I) (J) edge (L) (M) edge (O) (K) edge (N) (N) edge (Q) (Q) edge (T) (T) edge (W) (I) edge (K) (L) edge (N) (O) edge (Q) (R) edge (T) (P) edge (S) (S) edge (V) (N) edge (P) (Q) edge (S) (T) edge (V) (S) edge (U);


\end{tikzpicture}
\caption{The Auslander-Reiten quiver of the cluster category of Dynkin type $A_5$. The vertices have been replaced with values of the modified Caldero-Chapoton map $\rho$. Again, the dotted lines are identified with oppostite orientations.}
\label{fig:ca5quiver}
\end{figure}
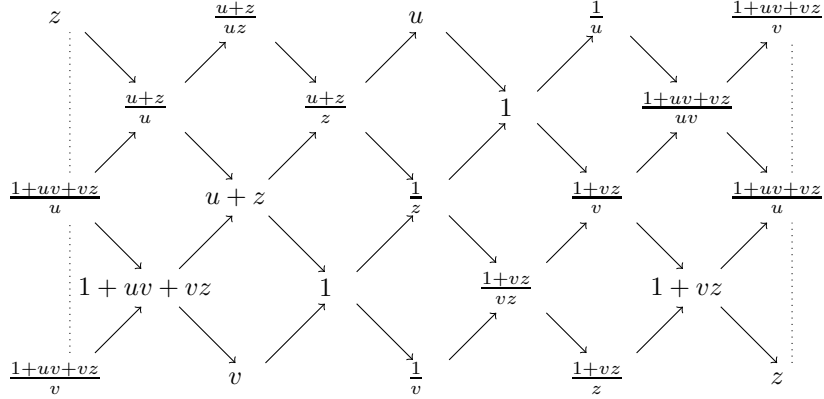

\par
This paper is organised as follows. Section \ref{sec:introductory} gives an essential background to the modified Caldero-Chapoton map, and explains some important results from \cite{friezes2}. Section \ref{sec:conditionf} introduces Condition F and proves Theorem \ref{alphabetadef}, whilst Section \ref{constructingf} shows how one can construct maps $\alpha$ and $\beta$ satisfying this condition. Section \ref{extension formula from friezes1} demonstrates how to find the multiplication formula for $\pi$, Section \ref{extensionadaption} adapts this formula to $\rho$ by proving the multiplication formula in Theorem \ref{thm:extensionadaption}, and Section \ref{sec:ca5example} shows why this formula is useful.
\par
It should be noted that Sections \ref{sec:introductory} and \ref{extension formula from friezes1} contain no original work, however they provide an essential setup for subsequent sections in the paper.

\section{A Modified Caldero-Chapoton Map: A Functorial Viewpoint} \label{sec:introductory}
In this section, we redefine the modified Caldero-Chapoton map in detail, using an equivalent, functorial viewpoint, allowing us simpler calculations throughout the rest of the paper. To set up,  we will follow the construction in \cite[sec.~2]{friezes2}. We add the assumption that $\mathsf{C}$ is 2-Calabi-Yau. We let $\mathsf{R}$ be a functorially finite subcategory of $\mathsf{C}$, which is closed under sums and summands and rigid; that is, $\textnormal{Hom}_{\mathsf{C}}(\mathsf{R},\Sigma \mathsf{R})=0$.
\par
We also assume that  $\mathsf{C}$ has a cluster tilting subcategory $\mathsf{T}$, belonging to a cluster structure in the sense of \cite[sec.~II.1]{clusterstructures}. We additionally require that $\mathsf{R}\subseteq\mathsf{T}$. Note that the Auslander-Reiten translation for $\mathsf{C}$ is
\[
\tau =\Sigma,
\]
and its Serre functor is $S=\Sigma^2$.
\par
There is a functor $G$, defined by:
\begin{equation}\label{def:thefunctorG}
\begin{split}
G: \mathsf{C} &\rightarrow \mathsf{Mod\,R}\\
c &\mapsto \textnormal{Hom}_{\mathsf{C}}(-,\Sigma c)|_{\mathsf{R}}
\end{split}
\end{equation}
where Mod\,$\mathsf{R}$ denotes the category of $\mathbb{C}$-linear contravariant functors $\mathsf{R} \rightarrow \mathsf{Vect}\,\mathbb{C}$. It is a $\mathbb{C}$-linear abelian category, and we denote by $\mathsf{fl\,R}$ the full subcategory of $\mathsf{Mod\,R}$, consisting of all finite length objects.
\par
The modified Caldero-Chapoton map is then defined as in Equation (\ref{eqn:introccmap}), where the sum is now taken over $e\in\,$K$_0(\mathsf{fl\,R})$. Here, we have
\begin{equation}\label{eqn:functorialalphabeta}
\alpha:\mathrm{obj}\,\mathsf{C}\rightarrow A~,~\beta:\mathrm{K}_0(\mathsf{fl\,R})\rightarrow A,
\end{equation}
 and these maps satisfy the exponential condition in Equation (\ref{eqn:alphabetaexponential})
\par
For two objects $a,b\in\mathsf{C}$ such that $Ga$ and $Gb$ have finite length, it is known by \cite[prop.~1.3]{friezes2} that $\rho$ is also exponential; that is,
\[
\rho(0)=1~,~\rho(a\oplus b)=\rho(a)\rho(b).
\]
It therefore suffices to calculate $\rho$ for each indecomposable object in $\mathsf{C}$.  Note that the formula for $\rho$ only makes sense when $Gc$ has finite length in $\mathsf{Mod\,R}$. That is, we require that $Gc\in\mathsf{fl\,R}$.

\section{Condition F on the maps $\alpha$ and $\beta$}\label{sec:conditionf}
We continue under the setup of Section \ref{sec:introductory}. Consider the exponential maps $\alpha$ and $\beta$ introduced earlier in Equation (\ref{eqn:functorialalphabeta}).
The following definition describes a canonical condition on $\alpha$ and $\beta$, which will later be used in Section \ref{extensionadaption} to prove a multiplication formula for $\rho$.
\begin{definition}{(Condition F).} \label{def:conditionf}
We say that the maps $\alpha$ and $\beta$ satisfy Condition F if, for each triangle
\begin{equation}
\label{xyztriangle}
\xymatrix{
x \ar[r]^{\varphi} & y \ar[r]^{\omega} & z \ar[r]^{\psi} & \Sigma x
}
\end{equation}
in $\mathsf{C}$, such that $Gx$, $Gy$ and $Gz$ have finite length in $\mathsf{Mod\,R}$, the following property holds:
\[
\alpha(y)=\alpha(x\oplus z)\beta([\textnormal{Ker}\,G\varphi]).
\]
\end{definition}
We now prove a theorem showing that in the case when $\mathsf{C}$ is 2-Calabi-Yau, Condition F can replace the definition of the frieze-like condition defined in \cite[def.~1.4]{friezes2}.

\begin{theorem}\label{alphabetadef}
Assume that the exponential maps $\alpha : \textnormal{obj}\,\mathsf{C}\rightarrow A$ and $\beta : $K$_0(\mathsf{fl\,R})\rightarrow A$ from Equation (\ref{eqn:functorialalphabeta}) satisfy Condition F from Definition~\ref{def:conditionf}. Then, the modified Caldero-Chapoton map $\rho$, defined in Equation~(\ref{eqn:introccmap}), is a generalised frieze in the sense that, where defined, it satisfies Equation (\ref{intro:def:genfrieze}), as well as the exponential conditions in (\ref{eqn:friezeexponential}).
\end{theorem}

\begin{proof}
Note that by \cite[thm.~1.6]{friezes2}, if 
\[
\xymatrix@R-1.8pc{
\Delta  =\big(  \Sigma c \ar[r]^-\xi & b\ar[r] & c\big)
}
\]
is an AR triangle in $\mathsf{C}$ such that $Gc$ and $G\Sigma c$ have finite length in $\mathsf{Mod\,R}$, and $\alpha$ and $\beta$ are frieze-like maps for $\Delta$ in the sense of \cite[def.~1.4]{friezes2}, then $\rho(\Sigma c)\rho(c)-\rho(b)\in\{0,1\}$. Hence, it suffices to show that $\alpha$ and $\beta$ are frieze-like for each $\Delta$.
We first note that if $c\notin\mathsf{R}\cup\Sigma^{-1}\mathsf{R}$, then $G(\Delta)$ is a short exact sequence, see \cite[lem.~1.12(iii)]{friezes1}. We now check the three cases of \cite[def.~1.4]{friezes2}.
\par
\textit{Case (i)}: Assume $c\notin\mathsf{R}\cup\Sigma^{-1}\mathsf{R}$ and $G(\Delta)$ is a split short exact sequence. That is,
\[
\xymatrix{
0 \ar[r] & G(\Sigma c) \ar[r]^{G\xi} & Gb \ar[r] & Gc \ar[r] & 0,
}
\]
is split short exact. It follows immediately that $G\xi$ has trivial kernel. Applying Condition F to $\Delta$, we obtain:
\begin{equation*}
\begin{split}
\alpha(b) & = \alpha(c\oplus\Sigma c)\beta([\textnormal{Ker}\,G\xi]) \\
& = \alpha(c\oplus\Sigma c)\beta(0)\\
& = \alpha(c\oplus\Sigma c),
\end{split}
\end{equation*}
where the final = is due to $\beta$ being exponential.
\par
\textit{Case (ii), first part}: Assume $c\notin\mathsf{R}\cup\Sigma^{-1}\mathsf{R}$ and $G(\Delta)$ is a nonsplit short exact sequence. Then, by the same working as in Case (i), we see that:
\[
\alpha(b)=\alpha(c\oplus\Sigma c).
\]
Now, consider the following triangle in $\mathsf{C}$:
\[
\xymatrix{
c \ar[r]^\nu & 0 \ar[r] & \Sigma c \ar[r]^{1_{\Sigma c}} & \Sigma c.
}
\]
Applying Condition F and using the fact that $\alpha$ is exponential, we obtain:
\begin{equation*}
\begin{split}
1 & = \alpha(0) \\
& = \alpha(c\oplus\Sigma c)\beta([\textnormal{Ker}\,G\nu]) \\
& = \alpha(c\oplus\Sigma c)\beta([Gc]).
\end{split}
\end{equation*}
We note that this manipulation works for any $c\in\,$obj$\,\mathsf{C}$.
\par
\textit{Case (ii), second part}: Let $c=\Sigma^{-1}r\in\Sigma^{-1}\mathsf{R}$. We showed in the first part of Case (ii) that
\[
\alpha(c\oplus\Sigma c)\beta([Gc])=1.
\]
Now, by \cite[lem.~1.12(i)]{friezes1}, $G(\Delta)$ becomes:
\[
G(\Delta)=0\rightarrow \textnormal{rad}\,P_r \rightarrow P_r , 
\]
where $P_r$ is the indecomposable projective Hom$_{\mathsf{R}}(-,r)$ in $\mathsf{Mod\,R}$, described in \cite[sec.~1.5]{friezes1}. Therefore, $G\xi$ has zero kernel, and applyng Condition F shows:
\begin{equation*}
\begin{split}
\alpha(b) & = \alpha(c\oplus\Sigma c)\beta([\textnormal{Ker}\,G\xi])\\
& = \alpha(c\oplus\Sigma c)\beta(0)\\
& = \alpha(c\oplus\Sigma c).
\end{split}
\end{equation*}
\par
\textit{Case (iii)}: Let $c=r\in\mathsf{R}$ and again consider the triangle
\[
\xymatrix{
c\ar[r]^\nu & 0 \ar[r] & \Sigma c \ar[r]^{1_{\Sigma c}} & \Sigma c,
}
\]
in $\mathsf{C}$. Applying Condition F, we see
\begin{equation*}
\begin{split}
1 & = \alpha(0) \\
& = \alpha(c\oplus\Sigma c)\beta([\textnormal{Ker}\,G\nu]) \\
& = \alpha(c\oplus\Sigma c)\beta(0) \\
& = \alpha(c\oplus\Sigma c),
\end{split}
\end{equation*}
where the third = is since $G(c)=G(r)=0$ and hence Ker$\,G\nu=0$. 
\newline
Now, by \cite[lem.~1.12(ii)]{friezes1}, applying $G$ to $\Delta$ gives the exact sequence:
\[
\xymatrix{
I_r \ar[r]^-{G\xi} & \textnormal{corad}\,I_r \ar[r] & 0,
}
\]
where $I_r$ is the indecomposable injective  Hom$_{\mathsf{R}}(-,\Sigma^2r)$ in $\mathsf{Mod\,R}$, described in \cite[sec.~1.10]{friezes1} and corad$\,I_r$ is its coradical. Additionally, by \cite[sec.~1.10]{friezes1}, we know that there is the following short exact sequence:
\[
0 \rightarrow S_r \rightarrow I_r \rightarrow \textnormal{corad}\,I_r \rightarrow 0,
\]
where $S_r$  in $\mathsf{Mod\,R}$ is the simple object supported at $r$, see \cite[prop.~2.2]{repartin2}. It follows that $\textnormal{corad}\,I_r\cong I_r/S_r$ and we see that Ker$\,G\xi=S_r$. So, by applying Condition F once again, we see that
\begin{equation*}
\begin{split}
\alpha(b) & = \alpha(c\oplus\Sigma c)\beta([\textnormal{Ker}\,G\xi])\\
& = \beta([\textnormal{Ker}\,G\xi])\\
& = \beta([S_r]),
\end{split}
\end{equation*}
where the second = is since $\alpha(c\oplus\Sigma c)=1$.
\end{proof}

\section{Constructing Maps that Satisfy Condition F} \label{constructingf}
We again continue under the setup of Section \ref{sec:introductory}. We will show that there exist maps $\alpha$ and $\beta$ satisfying Condition F, namely those given in \cite[def.~2.8]{friezes2}. Let us first look at the necessary constructions behind the definitions of $\alpha$ and $\beta$ in \cite[def.~2.8]{friezes2}.
\par
Recall that $\mathsf{T}$ is some cluster tilting subcategory of $\mathsf{C}$ with $\mathsf{R}\subseteq\mathsf{T}$, and denote by indec$\,\mathsf{T}$ the set of indecomposable objects in $\mathsf{T}$. For each $t\in\,$indec$\,\mathsf{T}$, one may find a unique indecomposable $t^*\in\,$indec$\,\mathsf{C}$, called the mutation of $t$, such that replacing $t$ with $t^*$ gives rise to another cluster tilting subcategory $\mathsf{T}^*$, see \cite[sec.~II.1]{clusterstructures}. Each such $t$ and $t^*$ fit into two exchange triangles (see \cite[sec.~II.1]{clusterstructures}):
\begin{equation}\label{exchangetriangles}
t^* \rightarrow a \rightarrow t~,~t\rightarrow a'\rightarrow t^*,
\end{equation}
where $a,a'\in\,$add$\,\big(($indec$\,\mathsf{T})\setminus t\big)$.
\par
We denote by K$_0^{\textnormal{split}}(\mathsf{T})$ the split Grothendieck group of the additive category $\mathsf{T}$ which has a basis formed by the set of indecomposables in $\mathsf{T}$. We note that K$_0^{\textnormal{split}}(\mathsf{T})$ carries the relations that $[a\oplus b]=[a]+[b]$, where $[a]$ is used to denote the K$_0^{\textnormal{split}}$-class of the object $a$ in $\mathsf{T}$.
\par
Define $\mathsf{S}$ to be the full subcategory of $\mathsf{C}$ which is closed under direct sums and summands and has
\begin{equation}\label{def:indecs}
 \mathrm{indec}\,\mathsf{S}=\mathrm{indec}\,\mathsf{T}\setminus\mathrm{indec}\,\mathsf{R},
\end{equation}
and then consider the subgroup
\begin{equation}\label{def:N}
N=
\left<
[a]-[a']~ \Bigg|
\begin{array}{l} s^*\rightarrow a\rightarrow s~,~ s\rightarrow a'\rightarrow s^*\, \textnormal{are exchange}\\
\textnormal{triangles with}\, s\in\textnormal{indec}\,\mathsf{S}
\end{array}
\right>
\end{equation}
of K$_0^{\textnormal{split}}(\mathsf{T})$, defined in \cite[def.~2.4]{friezes2}. Then, we denote by $Q$ the canonical surjection
\[
Q\,:\, \textnormal{K}_0^{\textnormal{split}}(\mathsf{T}) \twoheadrightarrow \textnormal{K}_0^{\textnormal{split}}(\mathsf{T})/N~,~Q([t])=[t]+N.
\]
\par
Now, for each $c\in\,$obj$\,\mathsf{C}$, we may construct the element ind$_{\mathsf{T}}(c)\in\,$K$_0^{\textnormal{split}}(\mathsf{T})$, called the index of c with respect to the cluster tilting subcategory $\mathsf{T}$. There exists a triangle $t'\rightarrow t\rightarrow c$ with $t,t'\in\mathsf{T}$ (see \cite[sec.~1]{dehykeller}). The index is then defined as:
\[
\textnormal{ind}_{\mathsf{T}}(c)=[t]-[t'],
\]
and is a well defined element of $\textnormal{K}_0^{\textnormal{split}}(\mathsf{T})$.
\par
Before giving the definition of $\alpha$ and $\beta$ from \cite[def.~2.8]{friezes2}, it remains to recall how to construct the homomorphism $\theta$ from \cite[sec.~2.6]{friezes2}. We do this by following the constructions in \cite[sec.~2.5]{friezes2}. Since $\mathsf{R}\subseteq\mathsf{T}$, the inclusion functor $i\,:\, \mathsf{R}\hookrightarrow\mathsf{T}$ induces the exact functor
\[
i^*:\,\mathsf{Mod\,T}\rightarrow \mathsf{Mod\,R},~ i^*(F)=F|_\mathsf{R},
\]
where $\mathsf{Mod\,T}$ is the abelian category of $\mathbb{C}$-linear contravariant functors $\mathsf{T}\rightarrow \mathsf{Vect}\,\mathbb{C}$. Now, by \cite[prop.~2.3(b)]{repartin2}, for each $t\in\,$indec$\,\mathsf{T}$, there is a simple object $\overline{S}_t\in\mathsf{Mod\,T}$ supported at $t$, and one may see that
\[
i^*\overline{S}_t=
\left\{
\begin{array}{cl}
S_t &\mathrm{if}~t\in\textnormal{indec}\,\mathsf{R},\\
0 & \mathrm{if}~t\in\textnormal{indec}\,\mathsf{S},
\end{array}
\right.
\]
where $S_t$ denotes the simple object in $\mathsf{Mod\,R}$ supported at $t$. Due to $i^*$ being exact, we can restrict it to the subcategories $\mathsf{fl\,T}$ and $\mathsf{fl\,R}$, made up of the finite length objects in $\mathsf{Mod\,T}$ and $\mathsf{Mod\,R}$, respectively. Then, there is an induced (surjective) group homomorphism
\[
\kappa :\, \textnormal{K}_0(\mathsf{fl\,T})\twoheadrightarrow\textnormal{K}_0(\mathsf{fl\,R}),
\]
with the obvious property that
\[
\kappa([\overline{S}_t])=
\left\{
\begin{array}{cl}
[S_t] & \mathrm{if}~t\in\textnormal{indec}\,\mathsf{R},\\
0 & \mathrm{if}~t\in\textnormal{indec}\,\mathsf{S}.
\end{array}
\right.
\]
\par
For the category $\mathsf{Mod\,T}$, there is a functor $\overline{G}$ similar to $G$ from Equation (\ref{def:thefunctorG}). It is defined by:
\begin{equation*}
\begin{split}
\overline{G}: \mathsf{C}& \rightarrow \mathsf{Mod\,T}\\
c & \mapsto \textnormal{Hom}_{\mathsf{C}}(-,\Sigma c)|_{\mathsf{T}}.
\end{split}
\end{equation*}
It is not hard to see that $\overline{G}$ has the property that $i^*\overline{G}=G$.
\par
We define $\theta$ to be the group homomorphism making the following diagram commute:
\begin{equation}\label{thetadiag}
\begin{gathered}
\xymatrix{
\textnormal{K}_{0}(\mathsf{fl\,T}) \ar@{->>}[d]^{\kappa} \ar[r]^-{\bar{\theta}} &\textnormal{K}_{0}^{\textnormal{split}}(\mathsf{T}) \ar@{->>}[d]^Q\\
\textnormal{K}_{0}(\mathsf{fl\,R}) \ar[r]_-\theta & \textnormal{K}_{0}^{\textnormal{split}}(\mathsf{T})/N,
}
\end{gathered}
\end{equation}
where 
\[
\bar{\theta}:\,\textnormal{K}_0(\mathsf{fl\,T})\rightarrow \textnormal{K}_0^{\textnormal{split}}(\mathsf{T})~,~\bar{\theta}([\overline{S}_t])=[a]-[a'],
\]
where $a,a'\in\,$add$\,\big(($indec$\,\mathsf{T})\setminus t\big)$ are from the exchange triangles for $t$ in (\ref{exchangetriangles}).
\par
Now, we may recall from \cite[def.~2.8]{friezes2} that the maps $\alpha\hspace{-1mm}:\,$obj$\,\mathsf{C}\rightarrow A$ and $\beta:\,$K$_0(\mathsf{fl\,R})\rightarrow A$ to a suitable ring $A$ can be defined by:
\begin{equation} \label{alphaandbeta}
\alpha(c)=\varepsilon Q(\textnormal{ind}_{\mathsf{T}}(c))~~\textnormal{and}~~\beta(e)=\varepsilon\theta(e),
\end{equation}
where $\varepsilon:\,$K$_0^{\textnormal{split}}(\mathsf{T})/N\rightarrow A$ is a suitably chosen exponential map, meaning that:
\begin{equation}\label{eqn:epsilonexponential}
\varepsilon(0)=1~,~\varepsilon(a+b)=\varepsilon(a)\varepsilon(b).
\end{equation}
Here, $a$ and $b$ denote two elements of K$_0^{\textnormal{split}}(\mathsf{T})/N$.

\begin{lemma}\label{alphabeta}
The maps $\alpha$ and $\beta$ from (\ref{alphaandbeta}) satisfy Condition F.
\end{lemma}

\begin{proof}
Consider the following triangle
\begin{equation}
\label{pf:xyztriangle}
\xymatrix{
x \ar[r]^{\varphi} & y \ar[r]^{\omega} & z \ar[r]^{\psi} & \Sigma x
}
\end{equation}
from (\ref{xyztriangle}). Then, by definition, we have
\begin{equation*}
\begin{split}
\alpha(y)&=\varepsilon Q(\textnormal{ind}_\mathsf{T}(y))\\
& = \varepsilon Q( \textnormal{ind}_\mathsf{T}(x)+\textnormal{ind}_\mathsf{T}(z)-\textnormal{ind}_\mathsf{T}(C)-\textnormal{ind}_\mathsf{T}(\Sigma^{-1}C))\\
& = \varepsilon(\textnormal{ind}_\mathsf{T}(x)+\textnormal{ind}_\mathsf{T}(z)-\textnormal{ind}_\mathsf{T}(C)-\textnormal{ind}_\mathsf{T}(\Sigma^{-1}C)+N)\\
& = \varepsilon(\textnormal{ind}_\mathsf{T}(x)+N)\varepsilon(\textnormal{ind}_\mathsf{T}(z)+N)\varepsilon(-\textnormal{ind}_\mathsf{T}(C)-\textnormal{ind}_\mathsf{T}(\Sigma^{-1}C)+N)\\
& = (*),
\end{split}
\end{equation*}
where $C$ in $\mathsf{C}$ is some lifting of Coker$\,\overline{G}(\Sigma^{-1}\omega)$ in the sense that $\overline{G}\Sigma^{-1}C=\,$Coker$\,\overline{G}(\Sigma^{-1}\omega)$. In the above manipulation, the second = is due to \cite[prop.~2.2]{Palu} and the penultimate = occurs since $\varepsilon$ is exponential, see Equation (\ref{eqn:epsilonexponential}).
\par
In addition,
\begin{equation*}
\begin{split}
\alpha(x\oplus z)\beta([\textnormal{Ker}\,G\varphi])& = \alpha(x)\alpha(z)\beta([\textnormal{Ker}\,G\varphi])\\
&=\varepsilon Q(\textnormal{ind}_\mathsf{T}(x))\varepsilon Q(\textnormal{ind}_\mathsf{T}(z))\beta([\textnormal{Ker}\,G\varphi])\\
& = \varepsilon(\textnormal{ind}_\mathsf{T}(x)+N)\varepsilon(\textnormal{ind}_\mathsf{T}(z)+N)\varepsilon \theta([\textnormal{Ker}\,G\varphi]),\\
& = (**)
\end{split}
\end{equation*}
where the first = is due to $\alpha$ being exponential and the penultimate = is just by the definition of $\beta$.
\par
Now, using the property that $i^*\overline{G}=G$,  it follows that
\[
[\textnormal{Ker}\,G\varphi]=[\textnormal{Ker}\,i^*\overline{G}\varphi]\overset{(1)}{=}[i^*\textnormal{Ker}\,\overline{G}\varphi]\overset{(2)}=\kappa[\textnormal{Ker}\,\overline{G}\varphi],
\]
where (1) follows from $i^*$ being an exact functor, and (2) from the definition of $\kappa$. We can now manipulate the expression ($**$) further:
\begin{equation*}
\begin{split}
(**)&=\varepsilon(\textnormal{ind}_\mathsf{T}(x)+N)\varepsilon(\textnormal{ind}_\mathsf{T}(z)+N)\varepsilon\theta(\kappa\big([\textnormal{Ker}\,\overline{G}\varphi])\big)\\
& = \varepsilon(\textnormal{ind}_\mathsf{T}(x)+N)\varepsilon(\textnormal{ind}_\mathsf{T}(z)+N)\varepsilon Q\big(\bar\theta([\textnormal{Ker}\,\overline{G}\varphi])\big)\\
& =\varepsilon(\textnormal{ind}_\mathsf{T}(x)+N)\varepsilon(\textnormal{ind}_\mathsf{T}(z)+N)\varepsilon\big(\bar\theta([\textnormal{Ker}\,\overline{G}\varphi])+N\big)\\
& = (***),
\end{split}
\end{equation*}
where the second equality is due to the commutativity of Diagram (\ref{thetadiag}).
\par
Comparing $(*)$ to $(***)$, we see that the required equality for Condition F is satisfied if
\begin{equation}
\label{resulteqn}
\bar\theta([\textnormal{Ker}\,\overline{G}\varphi])=-(\textnormal{ind}_\mathsf{T}(C)+\textnormal{ind}_\mathsf{T}(\Sigma^{-1}C)).
\end{equation}
Making use of the ``rolling" property on our triangle in (\ref{pf:xyztriangle}), we obtain the following sequence:
\[
\xymatrix{
\Sigma^{-1}y \ar[r]^{\Sigma^{-1}\omega} & \Sigma^{-1}z \ar[r]^-{\Sigma^{-1}\psi} & x \ar[r]^\varphi & y \ar[r]^\omega & z ,
}
\]
where any four consecutive terms form a triangle. Furthermore, since $\overline{G}$ is a homological functor, we may apply it to this sequence and produce the following long exact sequence in $\mathsf{fl\,T}$:
\begin{equation}
\label{longexactG}
\xymatrixcolsep{1.2pc}
\xymatrix{
& \overline{G}\Sigma^{-1}y \ar[rr]^{\overline{G}\Sigma^{-1}\omega} && \overline{G}\Sigma^{-1}z \ar[rr]^-{\overline{G}\Sigma^{-1}\psi} && \overline{G}x \ar[rr]^{\overline{G}\varphi} && \overline{G}y \ar[rr]^{\overline{G}\omega} && \overline{G}z.
}
\end{equation}
This shows Coker$\,\overline{G}\Sigma^{-1}\omega=\textnormal{Ker}\,\overline{G}\varphi$. Moreover, $C$ is chosen such that $\overline{G}\Sigma^{-1}C=\, $Coker$\,\overline{G}\Sigma^{-1}\omega$, and hence Ker$\,\overline{G}\varphi=\overline{G}\Sigma^{-1}C$. We can hence compute as follows:
\begin{equation*}
\begin{split}
\bar\theta([\textnormal{Ker}\,\overline{G}\varphi]) & = \bar\theta([\overline{G}\Sigma^{-1}C])\\
& = -(\textnormal{ind}_\mathsf{T}(\Sigma^{-1}C)+\textnormal{ind}_\mathsf{T}(\Sigma(\Sigma^{-1}C)))\\
&= -(\textnormal{ind}_\mathsf{T}(C)+\textnormal{ind}_\mathsf{T}(\Sigma^{-1}C)),
\end{split}
\end{equation*}
where the second = is due to \cite[lem.~2.10]{friezes2}. We can now see that Equation (\ref{resulteqn}) holds, and hence the lemma is proved.
\end{proof}

\begin{remark}
Through Lemma \ref{alphabeta} and Theorem \ref{alphabetadef}, we have managed to recover \cite[thm.~2.11]{friezes2}. Indeed, \cite[thm.~2.11]{friezes2} states that when
\[
\Delta=\Sigma c\rightarrow b\rightarrow c
\]
is an AR triangle in $\mathsf{C}$ such that $\overline{G}c$ and $\overline{G}(\Sigma c)$ have finite length in $\mathsf{Mod\,T}$, the maps $\alpha$ and $\beta$ from Equation (\ref{alphaandbeta}) satisfy the frieze-like condition given in \cite[def.~1.4]{friezes2}. By Lemma \ref{alphabeta}, we know that $\alpha$ and $\beta$ as defined in Equation (\ref{alphaandbeta}) satisfy Condition F. Theorem \ref{alphabetadef} proves that any $\alpha$ and $\beta$ satisfying Condition F also satisfy the frieze-like condition for $\Delta$ (recovering \cite[thm.~2.11]{friezes2}). Hence by \cite[thm.~1.6]{friezes2}, these $\alpha$ and $\beta$ turn $\rho$ into a generalised frieze.
\end{remark}

\section{The Multiplication Formula from \cite{friezes1}}\label{extension formula from friezes1}
In this section we demonstrate some of the technicalities behind the proof of the multiplication formula for $\pi$ from Equation (\ref{eqn:originalextensionintro}), proved in \cite[prop.~4.4]{friezes1}. This is done with a view of proving a similar formula for $\rho$ in Section \ref{extensionadaption}.  Following the setup of \cite[sec.~4]{friezes1}, for this section we do not require a cluster tilting subcategory $\mathsf{T}$, as the theory in \cite{friezes1} uses only the rigid subcategory $\mathsf{R}$. 
\par
Let $m\in\textnormal{indec}\,\mathsf{C}$ and $r\in\textnormal{indec}\,\mathsf{R}$ be indecomposable objects such that $\textnormal{Ext}^{1}_{\mathsf{C}}(r,m)$ and $\textnormal{Ext}^{1}_{\mathsf{C}}(m,r)$ both have dimension one over $\mathbb{C}$. As in \cite[rem.~4.2]{friezes1}, this allows us to construct the following nonsplit triangles in $\mathsf{C}$:
\begin{equation}
\label{nonsplit1}
\xymatrix{
m \ar[r]^\mu & a \ar[r]^\gamma & r \ar[r]^\delta & \Sigma m
}
\end{equation}
and
\begin{equation}
\label{nonsplit2}
\xymatrix{
r \ar[r]^\sigma & b \ar[r]^\eta & m \ar[r]^\zeta & \Sigma r,
}
\end{equation}
with $\delta$ and $\zeta$ nonzero. Note that ``rolling" the first triangle gives:
\[
\xymatrix{
\Sigma^{-1}r \ar[r]^-{-\Sigma^{-1}\delta} & m \ar[r]^{\mu} & a \ar[r]^\gamma & r,
}
\]
which is also a triangle in $\mathsf{C}$.
Applying the functor $G$ to both the ``rolled" triangle and the triangle in (\ref{nonsplit2}) gives the following exact sequences in $\mathsf{Mod\,R}$, obtained in \cite{friezes1}:
\[
\xymatrix{
G(\Sigma^{-1}r) \hspace{2mm} \ar[r]^{\hspace{3mm}-G(\Sigma^{-1}\delta)} & \hspace{2mm}Gm \ar[r]^{G\mu} & Ga \ar[r] & 0
}
\]
and
\[
\xymatrix{
0 \ar[r] &Gb \ar[r]^{G\eta} & Gm \ar[r]^{G\zeta} & G(\Sigma r).
}
\]
\begin{remark}
\begin{enumerate}
\item The zeros arise in each exact sequence due to $G(r)=\textnormal{Hom}(-,\Sigma r)|_{\mathsf{R}}$ being the zero functor. Indeed, since $\mathsf{R}$ is rigid, evaluating $G(r)$ at any $x$ in $\mathsf{R}$ will make the corresponding Hom-space zero.
\item The exact sequences are in $\mathsf{Mod\,R}$; that is, each term is a $\mathbb{C}$-linear contravariant functor $\mathsf{R} \rightarrow \mathsf{Vect}\,\mathbb{C}$.
\end{enumerate}
\end{remark}
 Letting Gr denote the Grassmannian of submodules, we have morphisms of algebraic varietes,
\[
\xymatrix{
\textnormal{Gr}(Ga) \ar@{^{(}->}[r]^\xi & \textnormal{Gr}(Gm) & \textnormal{Gr}(Gb) \ar@{_{(}->}[l]_\nu,
}
\]
\vspace{-3mm}
\[
\xymatrix{
\hspace{-1mm} P \ar@{|->}[r] & (G\mu )^{-1}(P), &
}
\]
\[
\xymatrix{
& \hspace{6mm} (G\eta) (N) & N.\ar@{|->}[l]
}
\]
It was proved in \cite[Lemma~4.3]{friezes1} that if $M$ in $\mathsf{Mod\,R}$ is some subfunctor of $Gm$, then either $M\subseteq \textnormal{Im}\,G\eta$ or $\textnormal{Ker}\,G\mu\subseteq M$, but not both. This means that for $M\subseteq Gm$ we can find either a subfunctor $N\subseteq Gb$ such that $(G\eta )(N)=M$ or we can find $P\subseteq Ga$ such that $(G\mu )^{-1}(P)=M$. Hence, the subfunctor $M$ is either of the form $(G\eta )(N)$ or $(G\mu )^{-1}(P)$.
\par
It is hence clear that Gr($Gm$) is isomorphic to the disjoint union of the images of $\xi$ and $\nu$. That is,
\[
\textnormal{Gr}(Gm)\cong \textnormal{Gr}(Gb) \bigsqcup \textnormal{Gr}(Ga).
\]
\begin{remark} \label{constructiblemaps}
\begin{enumerate}
\item We should note that for $M$ in $\mathsf{Mod\,R}$, the Grassmannian Gr$(M)$ is an algebraic variety. Therefore, it makes sense to calculate the Euler characteristic of Gr$(Gm)$, Gr$(Ga)$ and Gr$(Gb)$.
\item In addition, we note that $\xi$ and $\nu$ are both constructible maps, hence the images of $\xi$ and $\nu$ form constructible subsets in Gr$(Gm)$. See \cite[Section~2.1]{Palu} for the definitions of a constructible map and a constructible set.
\end{enumerate}
\end{remark}
The following statement then follows in \cite{friezes1}:
\begin{equation}\label{chiformula}
\chi (\textnormal{Gr}(Gm))=\chi (\textnormal{Gr}(Gb))+\chi (\textnormal{Gr}(Ga)),
\end{equation}
where $\chi$ again denotes the Euler characteristic defined by cohomology with compact support (see \cite[p.~93]{Fulton}). Using Remark \ref{constructiblemaps}, since the images of $\xi$ and $\nu$ are constructible sets inside Gr$(Gm)$, we know that $\chi$ is additive (see \cite[p.~92,~item~(3)]{Fulton}), which gives the above equality in (\ref{chiformula}). That is 
\[
\pi(m)=\pi(a)+\pi(b).
\]

\section{Adaptation of the Multiplication Formula to \cite{friezes2}} \label{extensionadaption}
This section builds on the material covered in the previous section and makes necessary adjustments and additions in order to obtain the multiplication formula for $\rho$, given in Theorem \ref{thm:extensionadaption}. Clearly, now that we are back working with $\rho$, we again require the setup of Section \ref{sec:introductory}; that is, we need a cluster tilting subcategory $\mathsf{T}$, with $\mathsf{R}\subseteq\mathsf{T}$.
\par
As with $\pi$, we look to understand how to evaluate $\rho$ for some $m\in\,$indec$\,\mathsf{C}$. In the definition of $\rho$, we take a sum over $e\in\,$K$_{0}(\mathsf{fl\,R})$. In order to do this, we will require knowledge of the Grothendieck group K$_{0}(\mathsf{fl\,R})$ and the K$_0$-classes of some of its key elements.
\par
Firstly, we know
\[
[\nu N]=[N].
\]
Indeed, by definition, $\nu N=(G\eta )(N)$, and since $G\eta$ is injective, $(G\eta )(N)$ and $N$ have the same composition series. Hence, the above equality is true. 
\par
To find $[\xi P]$, we first note that by definition, $\xi P = (G\mu )^{-1}P$, and therefore, $[\xi P]=[(G\mu)^{-1}P]$. A consequence of the Second Isomorphism Theorem is that a composition series of $(G\mu )^{-1}P$ can be obtained by concatenating composition series of $P$ and of Ker$\,G\mu$. That is, $[(G\mu )^{-1}(P)]=[P]+[$Ker$\,G\mu ]$.
So, the K$_0$-classes are:
\[
[\nu N]=[N]~\textnormal{and}~[\xi P]=[P]+[\textnormal{Ker}\,G\mu ].
\]
\par
Now that we have some useful information about the K$_0$-classes, we can take a more in depth look at $\rho$. Consider $r$ in $\mathsf{R}$, and let us calculate $\rho (r)$:
\begin{equation*}
\begin{split}
\rho (r) & =\alpha (r)\sum_{e}\chi \Big(\textnormal{Gr}_{e}(G(r))\Big)\beta (e) \\
& = \alpha (r)\sum_{e}\chi (\textnormal{Gr}_{e}(0))\beta (e) \\
& = \alpha (r)\beta (0) \\
& = \alpha (r).
\end{split}
\end{equation*}
In the above calculation, the third = is due to $\chi ($Gr$_{e}(0))$ being zero for all nonzero $e\in\,$K$_0(\mathsf{fl\,R})$ and one when $e=0$. The last = is since $\beta$ is exponential.
\par
Now consider $\rho (r)\rho(m)$ for $m\in\,$indec$\,\mathsf{C}$:
\begin{equation*}
\begin{split}
\rho (r)\rho(m) & = \alpha (r)\alpha (m)\sum_{e}\chi (\textnormal{Gr}_{e}(Gm))\beta (e)\\
& = \alpha (r)\alpha (m)\bigg( \sum_{e}\chi (\textnormal{Im}\,\xi \cap \textnormal{Gr}_{e}(Gm)) + \chi (\textnormal{Im}\,\nu\cap\textnormal{Gr}_{e}(Gm))\bigg)\beta (e),
\end{split}
\end{equation*}
where the second equality arises from Gr$(Gm)$ being the disjoint union of the images of $\xi$ and $\nu$.  We now make an important remark about the two intersections in the second equality above.

\begin{remark}
\begin{enumerate}
\item
The first intersection is given by the image of $\xi$ when applied to Gr$_e(Ga)$. Indeed,
\begin{equation*}
\begin{split}
\textnormal{Im}\,\xi\cap\textnormal{Gr}_{e}(Gm) & = \{\xi P\,|\,[\xi P]=e\}\\
& = \{\xi P\,|\,[P]=e-[\textnormal{Ker}\,G\mu ]\}\\
& = \xi (\textnormal{Gr}_{e-[\textnormal{Ker}G\mu]}(Ga)).
\end{split}
\end{equation*}
Here, we used the fact that $[\xi P]=[P]+[$Ker$G\mu]$.
\item 
The second intersection can be obtained in a similar way:
\begin{equation*}
\begin{split}
\textnormal{Im}\,\nu\cap\textnormal{Gr}_{e}(Gm) &=\{\nu N\,|\,[\nu N]=e\}\\
& = \{\nu N\,|\,[N]=e\}\\
& = \nu (\textnormal{Gr}_{e}(Gb))
\end{split}
\end{equation*}
\end{enumerate}
\end{remark}

Using this remark, we can continue to calculate $\rho (r)\rho (m)$, obtaining:
\begin{align}\label{rhorrhom}
\rho (r)\rho (m) & = \alpha (r)\alpha (m)\sum_{e}\bigg(\chi (\xi (\textnormal{Gr}_{e-[\textnormal{Ker}\,G\mu ]}(Ga)) + \chi (\nu (\textnormal{Gr}_{e}(Gb)))\bigg)\beta (e)
\nonumber \\
& = \alpha (r)\alpha (m)\sum_{e}\bigg(\chi (\textnormal{Gr}_{e-[\textnormal{Ker}\,G\mu ]}(Ga)) + \chi (\textnormal{Gr}_{e}(Gb))\bigg)\beta (e).
\end{align}
We can discard $\xi$ and $\nu$ in the final expression since they are both embeddings.

\begin{theorem} \label{thm:extensionadaption}
Let $m\in\textnormal{indec}\,\mathsf{C}$ and $r\in\textnormal{indec}\,\mathsf{R}$ such that $\textnormal{Ext}^{1}_{\mathsf{C}}(r,m)$ and $\textnormal{Ext}^{1}_{\mathsf{C}}(m,r)$ both have dimension one over $\mathbb{C}$. Then, there are nonsplit triangles 
\[
\xymatrix{
m \ar[r]^-\mu & a \ar[r]^-\gamma & r \ar[r]^-\delta & \Sigma m 
&and&
r \ar[r]^-\sigma & b \ar[r]^-\eta & m \ar[r]^-\zeta & \Sigma r,
}
\]
with $\delta$ and $\zeta$ nonzero. Let $Gm$ have finite length in $\mathsf{Mod\,R}$, then
\[
\rho(r)\rho(m)=\rho(a)+\rho(b).
\]
\end{theorem}

\begin{proof}
We first note that since $Gm$ has finite length, then so do $Ga$ and $Gb$. This follows immediately from \cite[rem.~4.2]{friezes1}.
\par
Now, by making the substitution $f=e-[\textnormal{Ker}\,G\mu]$ in Equation (\ref{rhorrhom}), observe that
\begin{equation*}
\begin{split}
\rho(r)\rho(m) &=\alpha(r)\alpha(m)\sum_f\chi\big(\textnormal{Gr}_f(Ga)\big)\beta(f+[\textnormal{Ker}\,G\mu])\\
& \hspace{8mm}+\alpha(r)\alpha(m)\sum_e\chi\big(\textnormal{Gr}_e(Gb)\big)\beta(e)\\
& \overset{(a)}{=} \alpha(r)\alpha(m)\sum_f\chi\big(\textnormal{Gr}_f(Ga)\big)\beta(f)\beta([\textnormal{Ker}\,G\mu])\\
& \hspace{8mm} +\alpha(r)\alpha(m)\sum_e\chi\big(\textnormal{Gr}_e(Gb)\big)\beta(e)\\
\end{split}
\end{equation*}
\begin{equation}
\label{rhorrhom1}
\begin{split}
&\hspace{15mm} \overset{(b)}{=} \alpha(r)\alpha(m)\beta([\textnormal{Ker}\,G\mu])\sum_f\chi\big(\textnormal{Gr}_f(Ga)\big)\beta(f)\\
& \hspace{23mm} +\alpha(r)\alpha(m)\sum_e\chi\big(\textnormal{Gr}_e(Gb)\big)\beta(e).
\end{split}
\end{equation}
Here, (a) is due to $\beta$ being exponential (see \cite[setup~1.1]{friezes2}) and (b) is due to $\beta([\textnormal{Ker}\,G\mu])$ being a constant.
\par
Now, consider Ker$\,G\sigma$. Since $G(r)=0$, then Ker$\,G\sigma=0$, and clearly $[\textnormal{Ker}\,G\sigma]=0$. The map $\beta$ is exponential, and therefore $\beta([\textnormal{Ker}\,G\sigma])=\beta(0)=1$. We can insert this in to Equation (\ref{rhorrhom1}) and see that:
\begin{equation}
\label{rhorrhom2}
\begin{split}
\rho(r)\rho(m) & = \alpha(r)\alpha(m)\beta([\textnormal{Ker}\,G\mu])\sum_f\chi\big(\textnormal{Gr}_f(Ga)\big)\beta(f)\\
&  \hspace{8mm}+\alpha(r)\alpha(m)\beta([\textnormal{Ker}\,G\sigma])\sum_e\chi\big(\textnormal{Gr}_e(Gb)\big)\beta(e).
\end{split}
\end{equation} 
Applying Condition F to our two triangles in the theorem, whilst remembering that $\alpha$ is exponential, gives
\begin{equation*}
\begin{split}
\alpha(a) & = \alpha(r)\alpha(m)\beta([\textnormal{Ker}\,G\mu]) \\
\alpha(b) & = \alpha(r)\alpha(m)\beta([\textnormal{Ker}\,G\sigma]).
\end{split}
\end{equation*}
Returning these equalities into Equation (\ref{rhorrhom2}), the expression for $\rho(r)\rho(m)$ becomes:
\begin{equation*}
\begin{split}
\rho(r)\rho(m) & =\alpha(a)\sum_f\chi\big(\textnormal{Gr}_f(Ga)\big)\beta(f)+\alpha(b)\sum_e\big(\textnormal{Gr}_e(Gb)\big)\beta(e)\\
& = \rho(a)+\rho(b).
\end{split}
\end{equation*}
\end{proof}

\section{Example for $\mathsf{C}(A_5)$}\label{sec:ca5example}

In this section, we will demonstrate the multiplication formula for $\rho$ in Theorem \ref{thm:extensionadaption} by recomputing a vertex in the AR quiver in Figure \ref{fig:ca5quiver}. We first give some brief background on the polygon model for $\mathsf{C}=\mathsf{C}(A_n)$, the cluster category of Dynkin type $A_n$.  Here, $n\geq 2$ is an integer. By \cite{clusterancase}, the indecomposables of $\mathsf{C}$ can be identified with the diagonals of a regular $(n+3)$-gon $P$ with the set of vertices $\{0,...,n+2\}$. By \cite[thm.~5.4]{friezes1}, the indecomposables of the rigid subcategory $\mathsf{R}\subseteq\mathsf{T}$ give a polygon dissection of $P$, and by \cite{clustercombinatorics} the cluster tilting subcategory $\mathsf{T}$ gives a full triangluation of the $(n+3)$-gon. Indeed, recall that there is a full subcategory $\mathsf{S}$, which is closed under direct sums and summands, such that
\[
\mathrm{indec}\,\mathsf{T}=\mathrm{indec}\,\mathsf{R}\cup\mathrm{indec}\,\mathsf{S}.
\]
Hence, the indecomposables in $\mathsf{S}$ correspond to a triangulation of each of the cells of $P$ given by the polygon dissection from indec$\,\mathsf{R}$. We should note in addition that each edge of the $(n+3)$-gon is identified with the zero object inside $\mathsf{C}$.
\par
This model also comes with the convenient property that for two indecomposables $a$ and $b$ in $\mathsf{C}$
\[
\mathrm{dim}_\mathbb{C}\mathrm{Ext}^1_\mathsf{C}(a,b)=
\left\{
\begin{array}{ll}
1, & \mathrm{if}~a~\mathrm{and}~b~\mathrm{cross}\\
0, & \mathrm{otherwise}.
\end{array}
\right.
\]
\par It is known by Theorem \ref{thm:extensionadaption} that for $m\hspace{-1mm}\in\,$indec$\,\mathsf{C}$ and $r\hspace{-1mm}\in\,$indec$\,\mathsf{R}$ such that dim$_\mathbb{C}$Ext$^1_\mathsf{C}(m,r)=\,$dim$_\mathbb{C}$Ext$^1_\mathsf{C}(r,m)=1$, then
\[
\rho(r)\rho(m)=\rho(a)+\rho(b),
\]
where $a$ and $b$ are the middle terms of the nonsplit extensions in (\ref{nonsplit1}) and (\ref{nonsplit2}). In the case of $\mathsf{C}=\mathsf{C}(A_n)$, $a$ and $b$ can be obtained as seen in the polygon in Figure \ref{fig:aandb}, where we have $a=a_1\oplus a_2$ and $b=b_1\oplus b_2$. See \cite[sec.~5]{friezes2} for details.
\par

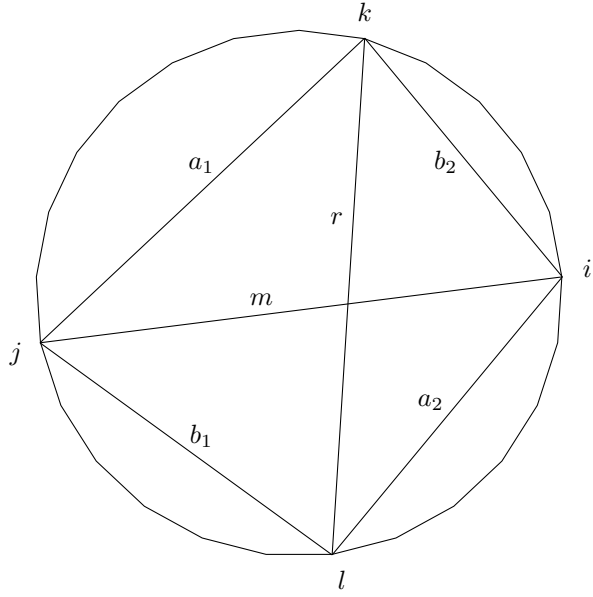
\begin{figure}
\centering
\begin{tikzpicture}
\node[regular polygon, regular polygon sides=25, draw, minimum size=7cm] (a){};

\path (a.corner 25) ++(90+25*360/25:10pt) node {$k$};
\path (a.corner 20) ++(90+20*360/25:10pt) node {$i$};
\path (a.corner 14) ++(90+14*360/25:10pt) node {$l$};
\path (a.corner 8) ++(90+8*360/25:10pt) node {$j$};
\draw (a.corner 25) --(a.corner 20);
\draw (a.corner 20) --(a.corner 14);
\draw (a.corner 14) --(a.corner 8);
\draw (a.corner 8) --(a.corner 25);
\draw (a.corner 8) --(a.corner 20);
\draw (a.corner 25) --(a.corner 14);

\node at (-0.5,-0.08) {$m$};
\node at (0.50,1) {$r$};
\node at (-1.3,1.7) {$a_1$};
\node at (1.75,-1.45) {$a_2$};
\node at (-1.3,-1.9) {$b_1$};
\node at (1.95,1.75) {$b_2$};
\end{tikzpicture}
\caption{There are nonsplit extensions $m\rightarrow a_1\oplus a_2\rightarrow r$ and $r\rightarrow b_1\oplus b_2 \rightarrow m$.}
\label{fig:aandb}
\end{figure}

\par
Now, to compute our example, we refer to the setup of \cite[sec.~3]{friezes2}; that is, we set $\mathsf{C}=\mathsf{C}(A_5)$, the cluster category of Dynkin type $A_5$. Thus, the indecomposables of $\mathsf{C}$ can be identified with the diagonals on a regular 8-gon. As in \cite{friezes2} we will denote by $\{a,b\}$ the indecomposable corresponding to the diagonal connecting the vertices $a$ and $b$.  We use the same polygon triangulation as in \cite[sec.~3]{friezes2}; that is, indec$\,\mathsf{R}$ corresponds to the red diagonals in Figure \ref{fig:8gon} and indec$\,\mathsf{S}$ corresponds to the blue diagonals. Hence, $\mathsf{T}$ contains the following indecomposable objects 
\[
\{1,7\}~,~\{2,4\}~,~\{2,5\}~,~\{2,7\}~,~\{5,7\},
\]
whilst the indecomposables in $\mathsf{S}$ are:
\[
\{1,7\}~,~\{2,4\}~,~\{5,7\}.
\]
These indecomposables in $\mathsf{S}$ fit in the following exchange triangles:
\[
\xymatrix@R=0.6pc{
\{1,7\}\ar[r]&\{2,7\}\ar[r]&\{2,8\}&\{2,8\}\ar[r]& 0\ar[r]&\{1,7\}\\
\{2,4\}\ar[r]&\{2,5\}\ar[r]&\{3,5\}&\{3,5\}\ar[r]& 0\ar[r]&\{2,4\}\\
\{5,7\}\ar[r]&\{2,5\}\ar[r]&\{2,6\}&\{2,6\}\ar[r]&\{2,7\}\ar[r]&\{5,7\}.\\
}
\]
Then, applying the definition of $N$ from (\ref{def:N}), it is easily seen that
\[
N=\big< [2,5],[2,7]\big>,
\]
Here, we denote by $[a,b]$ the K$_0^{\textnormal{split}}$-class of the indecomposable $\{a,b\}$. We also have
\[
\textnormal{K}_0^{\textnormal{split}}(\mathsf{T})/N=\Big<[1,7]+N,~ [2,4]+N,~[5,7]+N\Big>,
\]
and let the exponential map $\varepsilon:\,$K$_0^{\textnormal{split}}(\mathsf{T})/N \rightarrow \mathbb{Z}[u^{\pm1},v^{\pm1},z^{\pm1}]$ be given by:
\begin{equation}\label{def:epsilon}
\varepsilon([1,7]+N)=u~,~\varepsilon([2,4]+N)=v~,~\varepsilon([5,7]+N)=z.
\end{equation}

\begin{figure}
\centering
\begin{tikzpicture}[scale=2]
\draw (0.92,0.38) --(0.92,-0.38) --(0.38,-0.92) --(-0.38,-0.92) --(-0.92,-0.38) --(-0.92,0.38) --(-0.38,0.92) --(0.38,0.92) --(0.92,0.38);
\draw [blue, thick] (-0.38, 0.92) --(-0.92,-0.38);
\draw [blue, thick] (-0.38,-0.92) --(0.92,-0.38);
\draw [blue, thick] (0.38,0.92) --(0.92,-0.38);
\draw [red, thick] (-0.38,0.92) --(-0.38,-0.92);
\draw [red, thick] (-0.38,0.92) --(0.92,-0.38);

\node [above right] at (0.38,0.92) {1};
\node [above left] at (-0.38,0.92) {2};
\node [left] at (-0.95,0.38) {3};
\node [left] at (-0.95,-0.38) {4};
\node [below left] at (-0.38,-0.92) {5};
\node [below right] at (0.38, -0.92) {6};
\node [right] at (0.95,-0.38) {7};
\node [right] at (0.95,0.38) {8};
\end{tikzpicture}
\caption{Red diagonals correspond to indecomposables in indec$\,\mathsf{R}$ and blue diagonals correspond to indecomposables in indec$\,\mathsf{S}$.}
\label{fig:8gon}
\end{figure}
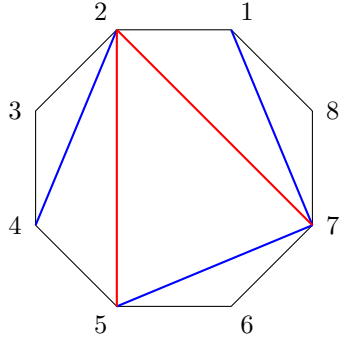

We will now demonstrate how to calculate $\rho(\{4,6\})$ using an alternative method to that in \cite[ex.~3.5]{friezes2}. We will compute it by applying the multiplication formula for $\rho$ in Theorem \ref{thm:extensionadaption}. Since dim$_\mathbb{C}$Ext$^1(\{4,6\},\{2,5\})=\,$dim$_\mathbb{C}$Ext$^1(\{2,5\},\{4,6\})=1$, we may set $r=\{2,5\}$, and using Figure \ref{fig:aandb}, we know that $\{4,6\}$ sits in the following nonsplit extensions:
\[
\{4,6\}\rightarrow\{2,4\}\rightarrow\{2,5\}~,~\{2,5\}\rightarrow\{2,6\}\rightarrow\{4,6\}.
\]
Applying Theorem \ref{thm:extensionadaption}, we get the following equality:
\begin{equation}\label{eqn:rhoex}
\rho(\{2,5\})\rho(\{4,6\})=\rho(\{2,4\})+\rho(\{2,6\}).
\end{equation}
Due to the fact that the diagonals corresponding to $\{2,5\}, \{2,4\}$ and $\{2,6\}$ do not cross any diagonals in indec$\,\mathsf{R}$, it is immediate that
\[
G(\{2,5\})=G(\{2,4\})=G(\{2,6\})=0.
\]
Hence, by the definition of $\rho$, Equation (\ref{eqn:rhoex}) becomes:
\begin{equation}\label{eqn:rhoex1}
\alpha(\{2,5\})\rho(\{4,6\})=\alpha(\{2,4\})+\alpha(\{2,6\}).
\end{equation}
In order to calculate $\alpha$ of each of the indecomposables $\{2,5\}, \{2,4\}$ and $\{2,6\}$, we first calculate their indices with respect to $\mathsf{T}$. 
\par
$\{2,5\}$ sits in the following triangle:
\[
0\rightarrow \{2,5\}\rightarrow\{2,5\},
\]
and since $\{2,5\}\in\,$indec$\,\mathsf{T}$, we see that:
\begin{equation}\label{ind25}
\textnormal{ind}_\mathsf{T}(\{2,5\})=[2,5].
\end{equation}
By the same logic,
\begin{equation}\label{ind24}
\textnormal{ind}_\mathsf{T}(\{2,4\})=[2,4].
\end{equation}
We note that one of the exchange triangles for $\{2,6\}$ is:
\[
\{5,7\}\rightarrow\{2,5\}\rightarrow\{2,6\},
\]
and hence
\begin{equation}\label{ind26}
\textnormal{ind}_\mathsf{T}(\{2,6\})=[2,5]-[5,7].
\end{equation}
\par
Since $[2,5]\in N$, using the definition of $\alpha$ from \cite[def.~2.8]{friezes2}, we see that:
\begin{equation*}
\begin{split}
\alpha(\{2,5\}) & =\varepsilon Q(\textnormal{ind}_\mathsf{T}(\{2,5\})\\
& = \varepsilon Q([2,5])\\
& = \varepsilon([2,5]+N)\\
& = \varepsilon(0)\\
&= 1,
\end{split}
\end{equation*}
 and hence, the right hand side of Equation (\ref{eqn:rhoex1}) becomes $\alpha(\{2,5\})\rho(\{4,6\})=\rho(\{4,6\})$. Substituting back into Equation (\ref{eqn:rhoex1}), we obtain:
\begin{equation*}
\begin{split}
\rho(\{4,6\}) &=\alpha(\{2,4\})+\alpha(\{2,6\})\\
& = \varepsilon Q(\textnormal{ind}_\mathsf{T}(\{2,4\}))+\varepsilon Q(\textnormal{ind}_\mathsf{T}(\{2,6\}))\\
& \overset{(1)}{=} \varepsilon Q([2,4])+\varepsilon Q([2,5]-[5,7])\\
& = \varepsilon([2,4]+N)+\varepsilon(-[5,7]+N)\\
& \overset{(2)}{=} v + z^{-1}\\
& = \frac{1+vz}{z},
\end{split}
\end{equation*}
where (1) is by substituting the values from Equations (\ref{ind24}) and (\ref{ind26}), and (2) is due to the definition of $\varepsilon$ from Equation (\ref{def:epsilon}). We notice here that this is indeed the same result for $\rho(\{4,6\})$ obtained in \cite[ex.~3.5]{friezes2}.

\begin{remark}
In general, the formula from Theorem \ref{thm:extensionadaption} can be applied iteratively to calculate $\rho(m)$. Indeed,
\[
\rho(r)\rho(m)=\rho(a)+\rho(b)
\] is an iterative formula on $m$, and hence, calculating $\rho$ of each indecomposable in $\mathsf{C}$ can be reduced to calculating $\rho$ of the indecomposables in $\mathsf{C}$ whose corresponding diagonals in the $(n+3)$-gon do not cross any of the diagonals in $\mathsf{R}$. Namely, it is clear from Figure \ref{fig:aandb} that each of $a_1,a_2,b_1$ and $b_2$ sit inside ``smaller" polygons than $m$. Here, the smaller polygons are those obtained from $r$ dissecting the (n+3)-gon. Since $\mathsf{R}$ consists of only non-crossing diagonals, the remaining diagonals in $\mathsf{R}$ sit inside these smaller polygons. Reapplying Theorem \ref{thm:extensionadaption} to each of $a_1,a_2,b_1$ and $b_2$ will again create a series of even smaller polygons, containing a new $a$ or $b$. After repeated iterations, this process will eventually terminate at the stage when the new $a$ or $b$ does not cross any of the diagonals in $\mathsf{R}$.
\par
Now, in the case when a diagonal, say $m'$, does not cross a diagonal in $\mathsf{R}$, calculating $\rho(m')$ is acheived by calculating $\alpha(m')$. This is clear since $G(m')=0$. Computing $\alpha(m')$ is done by calculating the index of $m'$, and then applying the maps $Q$ and $\varepsilon$. Hence, finding $\rho(m)$ for each $m\in\,$indec$\,\mathsf{C}$ can be reduced by Theorem \ref{thm:extensionadaption} to computing the index of each of the indecomposables in $\mathsf{C}$ whose corresponding diagonals do not cross any diagonals in $\mathsf{R}$.
\end{remark}

\noindent
{\textbf{Acknowledgement}.} The author acknowledges economic support in the form of a studentship from the School of Mathematics and Statistics at Newcastle University.

\end{document}